\documentclass[]{interact}

\usepackage{epstopdf}
\usepackage[caption=false]{subfig}
\usepackage[numbers,sort&compress]{natbib}
\bibpunct[, ]{[}{]}{,}{n}{,}{,}

\theoremstyle{plain}
\newtheorem{theorem}{Theorem}[section]
\newtheorem{lemma}[theorem]{Lemma}

\theoremstyle{definition}

\newtheorem{example}[theorem]{Example}

\theoremstyle{remark}
\newtheorem{remark}{Remark}

\usepackage[table]{xcolor}
\newcommand\norm[1]{\left\lVert#1\right\rVert}

\newcounter{RomanNumber}

\usepackage{algorithm}
\usepackage{algorithmic}
\usepackage{multirow}
\usepackage{xcolor}      

\usepackage{enumitem}

\begin{document}

\title{Nonuniqueness of Solutions of a Class of  $\ell_{0}$-minimization Problems}
\author{
\name{Jialiang Xu}
\thanks{CONTACT Jialiang Xu. Email: xujialiang@lsec.cc.ac.cn} 
\affil{Hua Loo-Keng Center for Mathematical Sciences, Academy of Mathematics and Systems Science, Chinese Academy of Sciences,  China}
}
\maketitle

\begin{abstract} 
Recently, finding the sparsest solution of an underdetermined linear system has become an important request in many areas such as compressed sensing, image processing, statistical learning, and data sparse approximation. 
In this paper, we study some theoretical properties of the solutions to a general class of $\ell_{0}$-minimization problems, which can be used to deal with many practical applications. We establish some necessary conditions for a point being the sparsest solution to this class of problems, and we also characterize the conditions for the multiplicity of the sparsest solutions to the problem. Finally, we discuss certain conditions for the boundedness of the solution set of this class of problems.
\end{abstract}

\begin{keywords}  
 $\ell_{0}$-minimization; Sparsity; Nonuniqueness; Boundedness. 
\end{keywords}

\section{Introduction}
 Let $\left \| x \right \|_{0}$ denote the number of nonzero components of the vector $x$ in this paper. We consider the  following $\ell_{0}$-minimization problem: 
\begin{equation}\label{Pnew}
\begin{array}{lcl}
(P_{0})&\min\limits_{x\in R^{n}}&\left \| x \right \|_{0}\\
& $s.t.$ & \left \| y-Ax \right \|_{2}\leqslant \epsilon,~ Bx\leqslant b,
\end{array} 
\end{equation}
 where $A\in R^{m\times n}$ and $B\in R^{l\times n}$ are two matrices with $m\ll n$ and  $l\leq n$, $y\in R^{m}$ and $b\in R^{l}$ are two given vectors,  $\epsilon\geq 0$ is a given parameter,  and $\left \| x \right \|_{2}=(\sum_{i=1}^{n}\left| x_{i}\right|^{2})^{1/2}$ is the $\ell_{2}$-norm of $x$. In compressed sensing (CS),  the parameter $\epsilon$ is often used to estimate the level of the  measurement error $e=y-Ax$.  Clearly,  the purpose of \eqref{Pnew} is to find the sparsest point in the convex set $T$ defined by
 \begin{equation*}\label{feasible set T}
T=\lbrace x: \left\| y-Ax\right\|_{2}\leqslant \epsilon, Bx\leqslant b\rbrace.
\end{equation*} 
The constraint $Bx\leq b$  is motivated by some practical applications which lets the  model \eqref{Pnew} be general enough to  cover several sparsity models including a few models widely used in compressed sensing \cite{D06,C06,CERT2006,DDEK11}, 1-bit compressed sensing \cite{gupta2010,laska2011, zhaochun2016}, and statistical regression \cite{tibs2007,hoef2010,rinaldo2009}. For example,  some structured sparsity models, including the nonnegative sparsity model  \cite{candes2005,CERT2006,Redbook,zhaobook2018} and the monotonic sparsity model  (isotonic regression) \cite{greenbook},  are   the special cases of the model \eqref{Pnew}.  Clearly, the following commonly used $\ell_{0}$-minimization  models   are also  the special cases of \eqref{Pnew}:
 \begin{equation*}
\begin{array}{ll}
$(C1)$~\min\limits_{x} \lbrace \Vert x\Vert_{0}:~ y=Ax\rbrace; & $(C2)$~ \min\limits_{x} \lbrace  \Vert x\Vert_{0}:~\left\| y-Ax\right\|_{2}\leq \varepsilon\rbrace.
\end{array}
\end{equation*}
The problems (C1)  and (C2) can be called the standard $\ell_{0}$-minimization problems \cite{Redbook,candes2005,zhaobook2018}. 

From theory to computation methods, an intensive study of (C1) has been carried out over the past decade. Some sufficient criteria have been developed for  the problem (C1) to have a unique sparsest solution, for  example, the criteria based on the  spark \cite{spark}, mutual coherence \cite{donoho2001}, null space property (NSP) \cite{cohen2009}, restricted isotonic property (RIP) \cite{candes2005}, exact recovery condition \cite{ERC}, and the range space property (RSP) \cite{zhao2013,zhao2014, zhaobook2018}.  Zhao also \cite{zhaounique} developed several other sufficient conditions for the uniqueness of the solution to the problem (C1), such as sub-mutual coherence, scaled mutual coherence, coherence rank and sub-Babel function. 

However, the above existing sufficient conditions are still very restrictive from a practical viewpoint. 
In practical signal recovery scenarios, the measured data is always inaccurate, in which case we use the sparsity model (C2) instead of (C1) or  more complex ones such as the model \eqref{Pnew}.  Different from  (C1), the model \eqref{Pnew} involves  a perturbation parameter $\epsilon$.   As a result, the uniqueness of the sparse solutions  of \eqref{Pnew} might not be guaranteed, and hence it also makes sense to understand the conditions under which the model has multiple sparsest solutions.  It is known that  an  $\ell_{1}$-minimization problem may solve (C1) under the NSP  \cite{cohen2009} and RIP assumptions \cite{candes2005} which ensures that the problem (C1) has a unique sparsest solution. However, Zhao \cite{zhao2013} has  shown that even if an underdetermined linear system admits multiple sparsest solutions, the $\ell_{1}$-minimization problem is still able  to  solve (C1) under a mild RSP assumption which does not necessarily require the uniqueness of the sparsest solution of the problem.  Therefore, in order to broadly understand the  property of $\ell_{0}$-problems, it is meaningful to identify some conditions under which the $\ell_{0}$-problem has multiple solutions. To this goal,  we characterize  the necessary conditions for a vector to be the sparsest solution of the problem, and sufficient conditions for the multiplicity of the solutions of \eqref{Pnew},  and  the condition for the solution set of \eqref{Pnew} to be bounded.

This paper is organized as follows. In Section \ref{property on saprse recovery}, we show some theoretical properties  of the problem \eqref{Pnew} such as the necessary conditions for a point being the sparsest solution to the problem \eqref{Pnew}. Section \ref{section mutiple solution} gives some sufficient conditions for the  nonuniqueness of the sparsest solutions of the problem \eqref{Pnew}. In Section \ref{the choice delta}, we develop some sufficient conditions for the boundedness of the solution set of the problem \eqref{Pnew}.  

$\bf{Notation:}$  The $\ell_{p}$-norm on $R^{n}$ is defined as $\left\| x\right\|_{p}=(\sum_{i=1}^{n}\left| x_{i}\right|^{p})^{1/p}$, where $p\geq 1$.    The field of real numbers is denoted by $R$ and the $n$-dimensional Euclidean space is denoted by $R^{n}$. 
The complementary set of $S\subseteq \left\{ 1,\dots,n\right\}$ with respect to $\{1,\dots, n\}$ is denoted by $\bar{S}$, i.e., $\bar{S}=\lbrace 1,\dots,n\rbrace \setminus S$. 
 For a given vector $x\in R^{n}$,   $x_{S}$ and $\vert x\vert$ denotes the vector supported on $S$ and  the vector with components $\vert x\vert_{j}=\vert x_{j}\vert$, $j=1,\dots,n$, respectively. Given a matrix $A$,   $a_{i,j}$ denotes the entry of $A$ in row $i$ and column $j$.  
  $A_{S}$ denotes the submatrix of $A\in R^{m\times n}$ obtained by deleting the columns indexed by $\bar{S}$, and $A_{I,S}$ denotes the submatrix of $A$ with components $a_{i,j}$ for $i\in I,~ j\in S$.

\section{Necessary conditions for the solutions of $(P_{0})$}\label{property on saprse recovery}

We first develop some necessary conditions for a point to be the solution of \eqref{Pnew}, which are summarized in the following Theorem \ref{sparse solution} and Theorem \ref{max matrix}.

\begin{theorem}\label{sparse solution}
If $x^{*}$ is the sparsest solution to \eqref{Pnew} where $A\in R^{m\times n}$ and $B\in R^{l\times n}$ are two matrices with columns $a_{i} ~ (i=1,2,\dots, n)$ and $b_{i}~ (i=1,2,\dots, n)$ respectively, then 
\begin{equation}\label{sp eq1}
\mathrm{Null}(A_{S})\cap \mathrm{Null}(B_{S})= \lbrace 0 \rbrace,
\end{equation}
  where  $S\subseteq \lbrace 1,2, \ldots, n\rbrace$ is the support set of $x^{*}$.
\end{theorem}
\begin{proof}
Let $x^{*}$ be the sparsest solution of \eqref{Pnew} and $k$ be the optimal value of \eqref{Pnew}. We  prove  this result by contradiction.  If $\mathrm{Null}(A_{S})\cap \mathrm{Null}(B_{S})\neq \lbrace 0 \rbrace$, there exists a nonzero vector $\Delta x\in R^{n}$ with $(\Delta x)_{S}\neq 0$ 
such that $A_{S}(\Delta x)_{S}=0$ and $B_{S}(\Delta x)_{S}=0, $ which can be written as 
\begin{equation*}\label{sparse sol eq2}
\sum_{i\in S} a_{i}(\Delta x)_{i}=0~\mathrm{and}~ \sum_{i\in S}b_{i}(\Delta x)_{i}=0.
\end{equation*}
Since $(\Delta x)_{S}\neq 0$, there is a nonzero component $(\Delta x)_{j}$,  $j\in S$, such that the corresponding   $a_{j}$ and  $b_{j}$ can be represented as the linear combination  of the other  columns,  that is,
\begin{equation}\label{sparse sol eq3}
a_{j}=-\sum_{i\in S, i\neq j} a_{i}\frac{(\Delta x)_{i}}{(\Delta x)_{j}}, ~~ b_{j}=-\sum_{i\in S, i\neq j} b_{i}\frac{(\Delta x)_{i}}{(\Delta x)_{j}}.
\end{equation}
Since  $x^{*}$ is feasible to the problem \eqref{Pnew}, we have 
$$\left\| y-\left(\sum_{i\in S, i\neq j}a_{i}x_{i}^{*}\right)-a_{j}x_{j}^{*}\right\|_{2}\leqslant \epsilon, ~ \left(\sum_{i\in S, i\neq j}b_{i}x_{i}^{*}\right)+b_{j}x_{j}^{*}\leqslant b.$$
Substituting $a_{j}$ and $b_{j}$ in \eqref{sparse sol eq3} into the  above system yields
\begin{equation}\label{sparse sol eq4}
\left\| y-\sum_{i\in S, i\neq j}\left(x_{i}^{*}-\frac{(\Delta x)_{i}}{(\Delta x)_{j}}x_{j}^{*}\right)a_{i}\right\|_{2}\leqslant \epsilon, ~ \sum_{i\in S, i\neq j}\left(x_{i}^{*}-\frac{(\Delta x)_{i}}{(\Delta x)_{j}}x_{j}^{*}\right)b_{i}\leqslant b.
\end{equation}
The inequalities in \eqref{sparse sol eq4} imply that the vector $\bar{x}$ with $\left\| \bar{x}\right\|_{0}\leqslant k-1$ defined as  
$$\bar{x}_{i}=\left\{\begin{matrix}
 x_{i}^{*}-\frac{(\Delta x)_{i}}{(\Delta x)_{j}}x_{j}^{*},~ & i\in S, i\neq j,\\ 
 0, ~ & i=j,\\
 0, ~ & i\notin S.
\end{matrix}\right.$$
is a feasible solution of \eqref{Pnew}.  This means  that  $\bar{x}$ is a  solution of \eqref{Pnew} sparser than $x^{*}$. This is a contradiction. The desired result follows. 
\end{proof}
Note that $\mathrm{Null}(A_{S})\cap \mathrm{Null}(B_{S})=\lbrace 0\rbrace$ means $\left[ \begin{array}{l}
  A  \\
  B  \\
\end{array} \right]_{S}$  has full column rank. We make the following comments for the condition $\mathrm{Null}(A_{S})\cap \mathrm{Null}(B_{S})= \lbrace 0 \rbrace$. 

\begin{remark}
 It can be seen that $\mathrm{Null}(A_{S})\cap \mathrm{Null}(B_{S})= \lbrace 0 \rbrace$ has some equivalent forms. Since $B_{S}x^{*}_{S}\leq b$ can be decomposed by active and inactive constraints, the following conditions can be regarded as the equivalent conditions for \eqref{sp eq1}:

    (i) $\mathrm{Null}\binom{A_{S}}{B_{\bar{I},S}}\cap \mathrm{Null}(B_{I,S})= \lbrace 0 \rbrace$; 

(ii) $\mathrm{Null}\binom{A_{S}}{B_{I,S}}\cap \mathrm{Null}(B_{\bar{I},S})= \lbrace 0 \rbrace$;

(iii) $\mathrm{Null}(A_{S})\cap \mathrm{Null}(B_{I, S})\cap \mathrm{Null}(B_{\bar{I},S})= \lbrace 0 \rbrace$.

\noindent Here $I\subseteq \lbrace 1,2, \ldots, m\rbrace$ is the index set of active constraints in $B_{S}x^{*}_{S}\leqslant b$ and $\bar{I}= \lbrace 1,2, \ldots, m\rbrace \setminus I$ is the index set of inactive constraints in $B_{S}x^{*}_{S}\leqslant b$.
\end{remark}
 Let $\vert I(x)\vert$ be the cardinality of active constraints in $Bx\leqslant b$ with respect to $x$. Denote the sparsest solution set by
\begin{equation}\label{sparse solution set}
\Lambda=\lbrace x\in R^{n}: \left\| x \right\|_{0}=k, ~x\in T\rbrace, 
\end{equation}
where $k$ is the optimal value of \eqref{Pnew}.  From the above remark, we see that the condition (\ref{sp eq1}) is equivalent to (ii) above. We may develop more specific necessary conditions than these conditions. For instance, in terms of maximum cardinality of $I(x),$ we can prove  the following result. 

\begin{theorem}\label{max matrix}
Let $x^{*}$ be a  solution to \eqref{Pnew} and $S$ be the support of $x^{*}$. If $x^{*}$ admits the maximum cardinality of  $ I(x)$,  $x\in \Lambda$, i.e., $\vert I(x^{*})\vert=\max\lbrace \vert I(x)\vert: x\in \Lambda \rbrace$, then 
\begin{equation}\label{Mstar}
M^{*}=\left[ \begin{array}{c}
 A_{S}\\ 
B_{I,S}
\end{array} \right]
\end{equation}
has  full column rank where  $I=I(x^{*})$.
\end{theorem} 
\begin{proof}
Let $x^{*}$ be a sparsest solution of \eqref{Pnew} which satisfies the assumption in  Theorem \ref{max matrix}. 
We  prove the result by contradiction. Assume  that  $\mathrm{Null}(M^{*})\neq \lbrace 0\rbrace$. Then there exists a nonzero vector $\Delta x$ with $(\Delta x)_{\bar{S}}=0$ and $(\Delta x)_{S}\neq 0$ such that 
\begin{equation}\label{max matrxi eq1}
A_{S}(\Delta x)_{S}=0~\mathrm{and}~ B_{I,S}(\Delta x)_{S}=0 .
\end{equation}
Then we construct a new vector $\bar{x}(\lambda)$ such that 
$$\bar{x}(\lambda)=x^{*}+\lambda (\Delta x)$$
where $\lambda $ is a  parameter.   Clearly, $\bar{x}(\lambda)$ continuously changes  with $\lambda$ and
\begin{equation}\label{maximun cardinality eq3}
\mathrm{supp}(\bar{x}(\lambda))\subseteq \mathrm{supp}(x^{*})~\mathrm{and}~\left\| \bar{x}(\lambda)\right\|_{0}\leqslant \left\| x^{*}\right\|_{0} 
\end{equation}
for all $\lambda$. If $\bar{x}(\lambda)$ satisfies the following system: 
\begin{equation}\label{maxmum cardinality eq2}
\left\| y-A_{S}z_{S}\right\|_{2}\leqslant \epsilon, ~ B_{I,S}z_{S}\leqslant  b_{I},~ B_{\bar{I},S}z_{S}\leqslant b_{\bar{I}},
\end{equation}
 then $\bar{x}(\lambda)$ is a feasible solution to \eqref{Pnew}, and hence $\bar{x}(\lambda)$ is a sparsest solution to \eqref{Pnew} which follows from \eqref{maximun cardinality eq3} and the fact that $x^{*}$ is a sparsest solution.
We now  prove that there exists a nonzero $\lambda$ such that $\bar{x}(\lambda)$ satisfies  the  system \eqref{maxmum cardinality eq2}.  Based on \eqref{max matrxi eq1}, the following two constraints are satisfied for all $\lambda$:
\begin{equation}\label{maximum cardinality eq1}
\left\| y-A_{S}\bar{x}_{S}(\lambda)\right\|_{2}\leqslant \epsilon, ~ B_{I,S}\bar{x}_{S}(\lambda)=b_{I}.
\end{equation}
 We only need  to check if $\bar{x}(\lambda)$ satisfies the third inequality in \eqref{maxmum cardinality eq2}.   
First we  denote three disjoint sets $J_{+}$, $J_{-}$, $J_{0}$ as follows, 
\begin{equation}\label{JJJ}
J_{+}=\lbrace j: ( B_{\bar{I},S}(\Delta x)_{S})_{j}>0\rbrace,~ J_{-}=\lbrace j: ( B_{\bar{I},S}(\Delta x)_{S})_{j}<0\rbrace, ~ J_{0}=\lbrace j: ( B_{\bar{I},S}(\Delta x)_{S})_{j}=0\rbrace.
\end{equation}
Consider the following  cases:

\begin{itemize}[leftmargin=-.05in,label=\color{blue}\theenumi]
\item[$\textbf{(M1)}$]
 $ J_{+}\cup J_{-}= \emptyset$. In this case  $B_{\bar{I},S}(\Delta x)_{S}=0$.  Combining with \eqref{max matrxi eq1} yields $(\Delta x)_{S}\in\mathrm{Null}(M^{*})\cap \mathrm{Null}(B_{\bar{I}, S})$. This contradicts to Theorem \ref{sparse solution}. Thus  we have only the next case.
\item[$\textbf{(M2)}$]
 $ J_{+}\cup J_{-}\neq \emptyset$. In this case $B_{\bar{I},S}(\Delta x)_{S}\neq 0$.   Let $\lambda\in [\lambda_{\min}, \lambda_{\max}]$ be  continuously increased from $\lambda_{\min}$ to $\lambda_{\max}$ where
$$\lambda_{\max}=\min_{j\in J_{+}}\biggr\lbrace\frac{(b_{\bar{I}}-B_{\bar{I},S}x^{*}_{S})_{j}}{( B_{\bar{I},S}(\Delta x)_{S})_{j}}\biggr\rbrace, ~ \lambda_{\min}=\max_{j\in J_{-}}\biggr\lbrace\frac{(b_{\bar{I}}-B_{\bar{I},S}x^{*}_{S})_{j}}{( B_{\bar{I},S}(\Delta x)_{S})_{j}}\biggr\rbrace.$$
Clearly, due to \eqref{JJJ},  $\lambda_{\min}<0$ and $\lambda_{\max}>0$.
For  $\lambda\in (0, \lambda_{\max}]$, we  have that:
$$(B_{\bar{I},S}\bar{x}_{S}(\lambda))_{i}\left\{\begin{matrix}
\leqslant (b_{\bar{I}})_{i} ,~ & i\in J_{+}, \\ 
 <(b_{\bar{I}})_{i}+\lambda*0=(b_{\bar{I}})_{i}, ~ & i\in J_{-},\\
 < (b_{\bar{I}})_{i}, ~& i\in J_{0}.
\end{matrix}\right.$$
The above second and third inequalities are obvious,  and the first inequality  follows from  the fact that for $i\in J_{+}$,
$$\begin{array}{lll}
(B_{\bar{I},S}\bar{x}_{S}(\lambda))_{i}&= (B_{\bar{I},S}x^{*}_{S})_{i}+\lambda(B_{\bar{I},S}(\Delta x)_{S})_{i} \\
 & \leqslant   (B_{\bar{I},S}x_{S}^{*})_{i}+\lambda_{\max} (B_{\bar{I},S}(\Delta x)_{S})_{i}\\
&\leqslant (B_{\bar{I},S}x^{*}_{S})_{i}+\frac{(b_{\bar{I}}-B_{\bar{I},S}x^{*}_{S})_{i}}{( B_{\bar{I},S}(\Delta x)_{S})_{i}}(B_{\bar{I},S}(\Delta x)_{S})_{i}=(b_{\bar{I}})_{i}.
\end{array} $$
For $\lambda\in[\lambda_{\min}, 0)$, we have that 
$$(B_{\bar{I},S}\bar{x}_{S}(\lambda))_{i}\left\{\begin{matrix}
 < (b_{\bar{I}})_{i}+\lambda*0=(b_{\bar{I}})_{i}, ~ & i\in J_{+},\\
 \leqslant (b_{\bar{I}})_{i}, ~& i\in J_{-},\\
 < (b_{\bar{I}})_{i}, ~ & i\in J_{0},
\end{matrix}\right.$$
where the second inequality follows from the fact that for $i\in J_{-}$,
$$\begin{array}{lll}
(B_{\bar{I},S}\bar{x}_{S}(\lambda))_{i}&\leqslant   (B_{\bar{I},S}x^{*}_{S})_{i}+\lambda_{\min} (B_{\bar{I},S}(\Delta x)_{S})_{i},\\
&\leqslant (B_{\bar{I},S}x^{*}_{S})_{i}+\frac{(b_{\bar{I}}-B_{\bar{I},S}x^{*}_{S})_{i}}{( B_{\bar{I},S}(\Delta x)_{S})_{i}} (B_{\bar{I},S}(\Delta x)_{S})_{i}=(b_{\bar{I}})_{i}.
\end{array} $$
 Note that $\bar{x}(\lambda)=x^{*}$ when $ \lambda=0$. Thus we have $$B_{\bar{I}, S}\bar{x}_{S}(\lambda)\leqslant b_{\bar{I}}$$ for all $\lambda\in [\lambda_{\min}, \lambda_{\max}]$.   Combining this with \eqref{maximum cardinality eq1}, we see that $\bar{x}(\lambda)\neq x^{*}$ for all $\lambda\in [\lambda_{\min}, \lambda_{\max}]$ satisfying \eqref{maxmum cardinality eq2}  and hence $\bar{x}(\lambda)$ is a feasible solution to \eqref{Pnew}.  Now starting from $ \lambda=0$, we continuously increase the value  $\vert\lambda\vert$.  
Thus, without loss of  generality, we assume  $\mathrm{supp}(\bar{x}(\lambda))=\mathrm{supp}(x^{*})$ when  $\vert\lambda\vert$ is increased continuously. Note that there exists a $\lambda^{*}\in [\lambda_{\min}, \lambda_{\max}]$ such that at least one index of inactive constraints in $B_{S}x^{*}_{S}\leq b$ will be added to the index set of active constraints in $B_{S}\bar{x}_{S}(\lambda^{*})\leqslant b$. That is, the index set of active  constraints in $B_{S}\bar{x}_{S}(\lambda^{*})\leqslant b$ includes $I$ and $D$:
$$ I(\bar{x}(\lambda^{*}))=I \cup D, ~\mathrm{where}~D=\lbrace j: (B_{\bar{I},S}\bar{x}_{S}(\lambda^{*}))_{j}=(b_{\bar{I}})_{j}\rbrace, ~  D \neq \emptyset.$$
This means $\vert I(\bar{x}(\lambda^{*}))\vert>\vert I(x^{*})\vert$ which contradicts the fact that $I(x^{*})$ has the maximum cardinality of $I(x)$ amongst all  sparsest solutions of \eqref{Pnew}. This contradiction shows that $M^{*}$ given in \eqref{Mstar} has full column rank.
\end{itemize}
\end{proof}

\section{ Multiplicity of sparsest solutions of  $(P_{0})$}\label{section mutiple solution}
 The  sparsest solutions of \eqref{Pnew} might not be unique  when the null space of $(A^{T},B^{T})^{T}$ is not reduced to the zero vector. In fact, any slight perturbation of the problem data $(A,B,b, y, \epsilon)$ may lead to the nonuniqueness of the solutions to the modified problem. This means that in most cases, the sparsest solutions for the problem \eqref{Pnew} are non-unique.  In this section, we  show that \eqref{Pnew} has infinitely many solutions under some mild conditions.
Let $x^{*}$ be a sparsest solution to \eqref{Pnew}. From Theorem \ref{sparse solution}, we  know that 
\begin{equation}\label{inf sol eq1}
\mathrm{Null}(A_{S})\cap \mathrm{Null}(B_{S})= \lbrace 0 \rbrace,
\end{equation}
which  can be separated into four cases: 
\begin{equation}\label{inf sol eq3}
\left\{\begin{matrix}
&\mathrm{Null}\binom{A_{S}}{B_{I,S}}\neq \lbrace 0\rbrace,~ \mathrm{Null}(B_{\bar{I},S})=\lbrace 0\rbrace,& \\ 
&\mathrm{Null}\binom{A_{S}}{B_{I,S}}\neq \lbrace 0\rbrace, ~  \mathrm{Null}(B_{\bar{I},S})\neq \lbrace 0\rbrace, &\mathrm{Null}\binom{A_{S}}{B_{S}}=\lbrace 0\rbrace,\\
 &\mathrm{Null}\binom{A_{S}}{B_{I,S}}=\lbrace 0\rbrace,~ \mathrm{Null}(B_{\bar{I},S})\neq \lbrace 0\rbrace,&\\
&\mathrm{Null}\binom{A_{S}}{B_{I,S}}=\lbrace 0\rbrace,~ \mathrm{Null}(B_{\bar{I},S})=\lbrace 0\rbrace,&
\end{matrix}\right.
\end{equation}
where $I$ and $\bar{I}$ are the index sets of active and inactive constraints in $B_{S}x_{S}^{*}\leqslant b$ respectively. 
Under some  conditions,  it  can be shown that for each case in \eqref{inf sol eq3},  \eqref{Pnew} has infinite sparsest solutions admitting the same support as that of the sparsest solution $x^{*},$ as indicated by the following Theorems \ref{infinite solution} and \ref{inf sol1}. Theorem \ref{infinite solution}  covers the first three cases and Theorem \ref{inf sol1} covers the last case in \eqref{inf sol eq3} respectively. 

\begin{theorem}\label{infinite solution}
Let $x^{*}$ be an arbitrary sparsest solution  to \eqref{Pnew} and $S$ be the support of $x^{*}$. The problem \eqref{Pnew} has infinitely many optimal solutions which have the same support as $x^{*}$ if the following condition  $(C1)$ holds:
\begin{itemize}[leftmargin=.2in]
\item {$ (C1)$} ~  $\mathrm{Null}\binom{A_{S}}{B_{I,S}}=\mathrm{Null}(M^{*})\neq\lbrace 0\rbrace$ and $x^{*}$ does not admit the maximum cardinality, i.e.,  $\vert I(x^{*})\vert\neq \max\lbrace\vert I(z)\vert:  z\in \Lambda\rbrace$ where $\Lambda$ is given in \eqref{sparse solution set}.
\end{itemize}
If  the corresponding error vector $e^{*}$, i.e., $e^{*}=y-Ax^{*}$, satisfies $\left\| e^{*}\right\|_{2}< \epsilon$, then \eqref{Pnew} has infinitely many optimal solutions which have the same support as $x^{*}$ if one of the following conditions $(C2)$, $(C3)$ and $(C4)$ holds:
\begin{itemize}[leftmargin=.2in]
\item{$( C2)$} $    ~ \mathrm{Null}(M^{*})=\lbrace 0\rbrace$ and $\mathrm{Null}(B_{S})\neq \lbrace 0 \rbrace$.
\item{$( C3)$} $  ~  \mathrm{Null}(M^{*})=\lbrace 0\rbrace$ and $\lbrace d: B_{I,S}d> 0\rbrace\cap \mathrm{Null}(B_{\bar{I},S})\neq \emptyset.$
\item{$( C4)$} $  ~  \mathrm{Null}(M^{*})=\lbrace 0\rbrace$ and $\lbrace d: B_{I,S}d< 0\rbrace\cap \mathrm{Null}(B_{\bar{I},S})\neq \emptyset.$
\end{itemize}
\end{theorem} 

\begin{proof}

$\textbf{(C1)}$  Consider the case $(C1)$ in Theorem \ref{infinite solution}.  We can find a nonzero $d$ such that $d_{S}\in \mathrm{Null}(M^{*})$ and $ d_{\bar{S}}=0$, leading to $$A_{S}d_{S}=0 ~ \mathrm{and} ~ B_{I,S}d_{S}=0.$$ Due to \eqref{inf sol eq1}, we know that $B_{\bar{I},S}d_{S}\neq 0$. Let $z(\lambda)$ be a vector which is constructed as $$z(\lambda)=x^{*}+\lambda d$$ where $\lambda$ is a  parameter.  It is easy to check that $z_{S}(\lambda)$ satisfies $$\left\| y-A_{S}z_{S}(\lambda)\right\|_{2}\leqslant \epsilon, ~B_{I,S}z_{S}(\lambda)=b_{I}.$$ 
Let the sets $J_{+}$, $J_{-}$ and $ J_{0}$ be still defined as the corresponding sets in \eqref{JJJ}  by replacing $(\Delta x)_{S}$ with $d_{S}$. Let $\lambda$ be restricted in $ [\lambda_{\min}, \lambda_{\max}]$  where 
$$\lambda_{\max}=\min_{j\in J_{+}}\lbrace\frac{(b_{\bar{I}}-B_{\bar{I},S}x^{*}_{S})_{j}}{( B_{\bar{I},S}d_{S})_{j}}\rbrace, ~ \lambda_{\min}=\max_{j\in J_{-}}\lbrace\frac{(b_{\bar{I}}-B_{\bar{I},S}x^{*}_{S})_{j}}{( B_{\bar{I},S}d_{S})_{j}}\rbrace.$$
 Similar to the case $(M2)$ in the proof of Theorem \ref{max matrix}, it can be proven that for all $\lambda\in [\lambda_{\min}, \lambda_{\max}]$, we have 
$B_{\bar{I},S}z_{S}(\lambda)\leqslant b_{\bar{I}}.$ Then $z(\lambda)$ is a feasible solution to \eqref{Pnew} when $ \lambda\in  [\lambda_{\min}, \lambda_{\max}]$, which together with the fact that $x^{*}$ is a sparsest solution and  $\mathrm{supp}(z(\lambda))\subseteq \mathrm{supp}(x^{*})$,  implies  for all $\lambda\in [\lambda_{\min}, \lambda_{\max}]$  $z(\lambda)$  is a sparsest solution of \eqref{Pnew} and hence $$\mathrm{supp}(x^{*})=\mathrm{supp}(z(\lambda)).$$ Since $z(\lambda)$ varies when $\lambda$ is changed continuously in the interval $[\lambda_{\min},  \lambda_{\max}]$, it implies that \eqref{Pnew} has infinitely many sparsest solutions with the same support as $x^{*}$.

 $\textbf{(C2)}$ Consider the case $(C2)$ in Theorem \ref{infinite solution}. We  choose a nonzero vector $\mu$  from the set $\mathrm{Null}(B_{S})$. Due to \eqref{inf sol eq1}, we have $A_{S}\mu\neq 0$. Let $t(\lambda)$ be a vector with components 
$$t_{S}(\lambda)=x^{*}_{S}+\lambda\mu, ~t_{\bar{S}}(\lambda)=0.$$
Then we have $B_{\bar{I},S}t_{S}(\lambda)< b_{\bar{I}} ~\mathrm{and}~ B_{I,S}t_{S}(\lambda)= b_{I}$ for all $\lambda$ which imply $B_{S}t_{S}(\lambda)\leqslant b.$
 Let $\vert\lambda\vert$ be restricted  in $(0, \lambda'_{\max}]$ with 
$$\lambda'_{\max}=\frac{\epsilon-\left\| e^{*}\right\|_{2}}{\left\| {A_{S}}\mu\right\|_{\infty}\sqrt{m}},$$
and $e^{*}=y-A_{S}x^{*}_{S}$. We have 
$$\begin{array}{lll}
\left\| y-A_{S}(x^{*}_{S}+\lambda\mu) \right\|_{2}&=&\left\| e^{*}-\lambda A_{S}\mu\right\|_{2},\\
 &\leqslant &\left\| e^{*}\right\|_{2}+\vert\lambda\vert\left\| A_{S}\mu\right\|_{2}\leqslant  \left\| e^{*}\right\|_{2}+\lambda'_{\max}\left\| A_{S}\mu\right\|_{2},\\
& = &\left\| e^{*}\right\|_{2}+\frac{\epsilon-\left\| e^{*}\right\|_{2}}{\sqrt{m}}\left\| A_{S}\mu/\left\| A_{S}\mu\right\|_{\infty}\right\|_{2},\\
 &\leqslant &\left\| e^{*}\right\|_{2}+\frac{\epsilon-\left\| e^{*}\right\|_{2}}{\sqrt{m}}\left\| \textbf{e}^{m} \right\|_{2}= \epsilon,\\
\end{array}$$
where the first inequality follows from the triangle inequality and $\textbf{e}^{m}$ is the vector of ones  with $m$ dimension.   Combining this with  the fact $B_{S}t_{S}(\lambda)\leqslant b$ implies that  $t(\lambda)$ is a feasible  solution of \eqref{Pnew} when $\lambda \in[0, \lambda'_{\max}]$. Same as the proof in $(C1)$,  it implies that $t(\lambda)$ is the sparsest solution of \eqref{Pnew} when $\lambda \in[0, \lambda'_{\max}]$, and hence  we obtain the desired result. Moreover,   the active and inactive indices in $Bt(\lambda)\leqslant b$ are the same as that in  $Bx^{*}\leqslant b$.

 $\textbf{(C3)}$ Consider the case $(C3)$ in Theorem \ref{infinite solution}.  We can find a nonzero vector $\xi$ from the set $\lbrace d: B_{I,S}d> 0\rbrace\cap \mathrm{Null}(B_{\bar{I},S})$ satisfying
$$ B_{\bar{I},S}\xi= 0~\mathrm{and}~ B_{I,S}\xi> 0.$$
Since the two cases $A_{S}\xi=0$ and $A_{S}\xi\neq 0$ do not contradict  $\mathrm{Null}(M^{*})=\lbrace 0\rbrace$,  we consider both of them.  
Let $v(\lambda)$ be a vector with components $$v_{S}(\lambda)=x^{*}_{S}+\lambda\xi ~ \mathrm{and}~ v_{\bar{S}}(\lambda)=0,$$ where $\lambda$ is a  parameter. Clearly, $\mathrm{supp}(v(\lambda))\subseteq \mathrm{supp}(x^{*})$ for $\lambda$. Now we claim that $v(\lambda)$ is a sparsest solution to \eqref{Pnew} in both cases of  $A_{S}\xi=0$ and  $A_{S}\xi\neq 0$ when $\lambda$ is restricted in certain interval.

{\bf 1)}   $A_{S}\xi\neq 0. $   When $\lambda\in [-\lambda''_{\max},0)$ with $\lambda''_{\max}=\frac{\epsilon-\left\| e^{*}\right\|_{2}}{\left\| {A_{S}}\xi\right\|_{\infty}\sqrt{m}}$, by the same  proof as in $(C2)$, we have $$\left\| y-A_{S}v_{S}(\lambda)\right\|_{2}\leqslant \epsilon.$$ It is easy to check that $$B_{I,S}v_{S}(\lambda)<b_{I} ~ \mathrm{and} ~ B_{\bar{I},S}v_{S}(\lambda)<b_{\bar{I}}.$$ Thus  $v(\lambda)$ is a feasible point  in $T$ for all $\lambda\in [-\lambda''_{\max},0]$. $\mathrm{supp}(v(\lambda))\subseteq \mathrm{supp}(x^{*})$ and the fact that $x^{*}$ is a sparsest point in $T$ imply that $v(\lambda)$ is a sparsest point in $T$ when $\lambda\in [-\lambda''_{\max},0]$.

{\bf 2)}  $A_{S}\xi= 0.   $   Here $\lambda$ can be any negative number so that $v(\lambda)$ is  a feasible point  in $T$. Similarly, $v(\lambda)$ is a sparsest solution to \eqref{Pnew} when $\lambda\leq 0$. Combining  $1)$ and $2)$ implies the desired result.

 $\textbf{(C4)}$ This proof is omitted. Note that $\lbrace d: B_{I,S}d> 0\rbrace\cap \mathrm{Null}(B_{\bar{I},S})\neq \emptyset$ is equivalent to  $\lbrace d: B_{I,S}d< 0\rbrace\cap \mathrm{Null}(B_{\bar{I},S})\neq \emptyset.$ Thus we can directly get the desired result.
\end{proof}
 It follows from  Theorem  \ref{max matrix} that   the linear dependence of the columns of $M^{*}$ implies that  $I(x^{*})$ does not have the maximum cardinality  amongst $I(x), x\in \Lambda$.  Therefore  the condition in $(C1)$ is mild. Note that the case $(C1)$ corresponds to the first two cases in \eqref{inf sol eq3}, and the cases $(C2)-(C4)$ correspond to the third case in \eqref{inf sol eq3}. Now we consider the last case in \eqref{inf sol eq3} and have the following theorem.
\begin{theorem}\label{inf sol1}
Let $x^{*}$ be an arbitrary sparsest solution  of \eqref{Pnew}, $S$ be the support of $x^{*}$. Assume that $\mathrm{Null}(M^{*})=\lbrace 0\rbrace$ and $\mathrm{Null}(B_{\bar{I},S})=\lbrace 0\rbrace$. Then \eqref{Pnew} has infinitely many optimal solutions with the same support as $x^{*}$  if  one of the following conditions  holds:
\begin{itemize}[leftmargin=0.2in]
  \item{$(D1)$}  $\lbrace d: B_{I,S}d> 0\rbrace\cap  \lbrace d: A_{S}d= 0\rbrace\neq \emptyset.$
\item{$(D2)$} $\lbrace d: B_{I,S}d< 0\rbrace\cap  \lbrace d: A_{S}d= 0\rbrace\neq \emptyset.$
\end{itemize}
If  the corresponding error vector $e^{*}$, i.e., $e^{*}=y-Ax^{*}$, satisfies $\left\| e^{*}\right\|_{2}< \epsilon$, then \eqref{Pnew} has infinitely many optimal solutions which have the same support as $x^{*}$ if one of the following conditions  holds:
\begin{itemize}[leftmargin=.2in]
\item{$(D3)$}  $ \mathrm{Null}(B_{I,S})\neq \lbrace 0 \rbrace.$
 \item{$(D4)$} $\lbrace d: B_{I,S}d> 0\rbrace\cap  \lbrace d: A_{S}d\neq 0\rbrace\neq \emptyset.$
\item{$(D5)$}  $\lbrace d: B_{I,S}d< 0\rbrace\cap  \lbrace d: A_{S}d\neq 0\rbrace\neq \emptyset.$
\end{itemize}
\end{theorem}
\begin{proof}
We start from $(D3)$.\\

$\textbf{(D3)}$ Since  $\mathrm{Null}(M^{*})=\lbrace 0\rbrace$ and $\mathrm{Null}(B_{I,S})\neq \lbrace 0\rbrace$,   for  $\forall \bar{d}\in \mathrm{Null}(B_{I,S})$, we have  
$$B_{I,S}\bar{d}=0~\mathrm{and}~ A_{S}\bar{d}\neq 0.$$
Since $\mathrm{Null}(B_{\bar{I},S})=\lbrace 0\rbrace$,  we have $B_{\bar{I},S}\bar{d}\neq 0$. Denote $$G_{0}=\lbrace j: (B_{\bar{I},S}\bar{d})_{j}=0\rbrace, ~G_{-}=\lbrace j: (B_{\bar{I},S}\bar{d})_{j}<0\rbrace, ~G_{+}=\lbrace j: (B_{\bar{I},S}\bar{d})_{j}>0\rbrace.$$ Clearly,  $G_{+}\cup G_{-}\neq \emptyset$. Let $\bar{z}(\lambda)$ be a vector with components $$\bar{z}_{S}(\lambda)=x^{*}_{S}+\lambda \bar{d} ~\mathrm{and} ~ \bar{z}_{\bar{S}}(\lambda)=0.$$ Clearly, $\mathrm{supp}(\bar{z}(\lambda))\subseteq \mathrm{supp}(x^{*})$ for all $\lambda$. Let $\vert \lambda\vert$ be restricted in $ (0, \min(\lambda_{1}, \lambda_{2})]$ where 
$$\lambda_{1}=\min_{j\in G_{+}\cup G_{-}}\frac{(b_{\bar{I}}-B_{\bar{I}, S}x_{S}^{*})_{j}}{  \vert (B_{\bar{I},S}\bar{d})\vert_{j}}, ~ \lambda_{2}=\frac{\epsilon-\left\| e^{*}\right\|_{2}}{\left\| A_{S}\bar{d}\right\|_{\infty}\sqrt{m}}.$$ For $i\in G_{+}\cup G_{-}$, 
$$\begin{array}{llll}
(B_{\bar{I},S}\bar{z}_{S}(\lambda))_{i}&=&(B_{\bar{I},S}x_{S}^{*})_{i}+\lambda (B_{\bar{I},S}\bar{d})_{i}\\
 & \leqslant & (B_{\bar{I},S}x_{S}^{*})_{i}+\vert\lambda\vert \vert (B_{\bar{I},S}\bar{d})_{i}\vert\\
&\leqslant & (B_{\bar{I},S}x_{S}^{*})_{i}+\lambda_{1} \vert (B_{\bar{I},S}\bar{d})_{i}\vert \\
& \leqslant &  (B_{\bar{I},S}x_{S}^{*})_{i}+\frac{(b_{\bar{I}}-B_{\bar{I}, S}x_{S}^{*})_{i}}{  \vert (B_{\bar{I},S}\bar{d})\vert_{i}} \vert (B_{\bar{I},S}\bar{d})_{i}\vert  =  (b_{\bar{I}})_{i}.
\end{array}$$
The above fact, combined with  $(B_{\bar{I},S}\bar{z}_{S}(\lambda))_{i}<(b_{\bar{I}})_{i},~ i\in G_{0}$, implies that $B_{\bar{I},S}\bar{z}_{S}(\lambda)\leqslant b_{\bar{I}}$. We also have $\left\| y-A_{S}\bar{z}_{S}(\lambda)\right\|_{2} \leqslant \epsilon$ which has been proven for many times in Theorem \ref{infinite solution}. These, combined with the fact that $B_{I,S}\bar{z}_{S}(\lambda)=b_{I}$, implies that 
$\bar{z}(\lambda)$ is a sparsest point in $T$ with the same support as $x^{*}$ when $\lambda\in [0, \min(\lambda_{1}, \lambda_{2})]$.

$\textbf{(D4)}$ Clearly, there exists a nonzero vector $d'$ such that 
$$B_{I,S}d'>0, ~ A_{S}d'\neq 0.$$
Since $\mathrm{Null}(B_{\bar{I},S})=\lbrace 0\rbrace$, we have $B_{\bar{I},S}d'\neq 0$.
Denote $$J'_{0}=\lbrace j: (B_{\bar{I},S}d')_{j}=0\rbrace,  ~ J'_{-}=\lbrace j: (B_{\bar{I},S}d')_{j}<0\rbrace, ~ J'_{+}=\lbrace j: (B_{\bar{I},S}d')_{j}>0\rbrace.$$ Clearly, $J'_{+}\cup J'_{-}\neq \emptyset$. Let $z'(\lambda)$ be a vector with components $z'_{S}(\lambda)=x^{*}_{S}+\lambda d'$ and $z'_{\bar{S}}(\lambda)=0.$ 
Let $\lambda$ be restricted in  $ [ \max(\lambda'_{1}, \lambda'_{2}),0)$ where 
$$\lambda'_{1}=\max_{j\in  J'_{-}}\frac{(b_{\bar{I}}-B_{\bar{I}, S}x_{S}^{*})_{j}}{   (B_{\bar{I},S}d')_{j}}, ~ \lambda'_{2}=-\frac{(\epsilon-\left\| e^{*}\right\|_{2})}{\left\| A_{S}d'\right\|_{\infty}\sqrt{m}}.$$ 
For $i\in J'_{-}$, we have 
$$\begin{array}{llll}
(B_{\bar{I},S}z'_{S}(\lambda))_{i}&=&(B_{\bar{I},S}x_{S}^{*})_{i}+\lambda (B_{\bar{I},S}d')_{i}\leqslant  (B_{\bar{I},S}x_{S}^{*})_{i}+\lambda'_{1} (B_{\bar{I},S}d')_{i},\\
&\leqslant &(B_{\bar{I},S}x_{S}^{*})_{i}+\frac{(b_{\bar{I}}-B_{\bar{I}, S}x_{S}^{*})_{i}}{   (B_{\bar{I},S}d')_{i}}  (B_{\bar{I},S}d')_{i}=  (b_{\bar{I}})_{i}.\\
\end{array}$$
For $i\in J'_{+}\cup J'_{0}$, we have $(B_{\bar{I},S}z'_{S}(\lambda))_{i}<(b_{\bar{I}})_{i}. $ It can be proven that $\left\| y-A_{S}z'_{S}(\lambda)\right\|_{2} \leqslant \epsilon$ for $\lambda\in [\max(\lambda'_{1}, \lambda'_{2}),0)$, which combined with  the fact $B_{I,S}z'_{S}(\lambda)<b_{I}$ implies that $z'(\lambda)$ is a sparsest point in $T$ with the same support as $x^{*}$ when $ \lambda \in[ \max(\lambda'_{1}, \lambda'_{2}),0]$, i.e., $\mathrm{supp}(x^{*})=\mathrm{supp}(z'(\lambda))$. \\
$\textbf{(D1)}$ Clearly,  there exists a nonzero vector $d''$ such that 
$$B_{I,S}d''>0, ~ A_{S}d''= 0.$$
Since $\mathrm{Null}(B_{\bar{I},S})=\lbrace 0\rbrace$, we have $B_{\bar{I},S}d''\neq 0$.
Denote $$J''_{0}=\lbrace j: (B_{\bar{I},S}d'')_{j}=0\rbrace, ~ J''_{-}=\lbrace j: (B_{\bar{I},S}d'')_{j}<0\rbrace, ~ J''_{+}=\lbrace j: (B_{\bar{I},S}d'')_{j}>0\rbrace.$$ Clearly, $J''_{+}\cup J''_{-}\neq \emptyset$. Let $z''(\lambda)$ be a vector with components $$z''_{S}(\lambda)=x^{*}_{S}+\lambda d''~\mathrm{and}~ z''_{\bar{S}}(\lambda)=0.$$ 
Due to $A_{S}d''=0$,  $\left\|y-A_{S}z''_{S}(\lambda)\right\|_{2}\leqslant \epsilon$ is satisfied. Let $\lambda$ be restricted in $ [\lambda''_{1},0)$ where 
$$\lambda''_{1}=\max_{j\in  J''_{-}}\frac{(b_{\bar{I}}-B_{\bar{I}, S}x_{S}^{*})_{j}}{    (B_{\bar{I},S}d'')_{j}}.$$ 
Similar to the proof of  $B_{\bar{I},S}z'_{S}(\lambda)<b_{\bar{I}}$ in (D4), we have $B_{\bar{I},S}z''_{S}(\lambda)<b_{\bar{I}}$.
The fact $B_{I,S}z''_{S}(\lambda)<b_{I}$ and $\left\| y-A_{S}z''_{S}(\lambda)\right\|_{2} \leqslant \epsilon$ implies that $z''(\lambda)$ is a sparsest point in $T$ with the same support as $x^{*}$ when $\lambda\in [\lambda''_{1},0]$, i.e., $\mathrm{supp}(x^{*})=\mathrm{supp}(z''(\lambda))$.

$\textbf{(D2,5)}$ The proof is omitted. Note that (D2)  is equivalent to (D1)  and that (D5) is equivalent to (D4). Thus the desired results can be obtained immediately. 
\end{proof}

Through the above theoretical analysis, we know that \eqref{Pnew} may have  infinitely many  sparsest solutions. We also want to know whether the sparsest solution set $\Lambda$ given in \eqref{sparse solution set} is bounded or not. This question will be explored in Section \ref{the choice delta}. The example below is given to illustrate the results of Theorems \ref{infinite solution} and \ref{inf sol1}.

\begin{example}\label{example 1}
Consider the  system $\left\| y-Ax\right\|_{2}\leqslant \epsilon, ~Bx\leqslant b$ with $\epsilon=10^{-1}$, where
{\small $$ A=\left[ \begin{array}{cccc}
1 & 0 & -2 &5 \\ 
 0&  1&  4& -9\\ 
 1&  0&  -2&5 
\end{array} \right], ~B=\left[ \begin{array}{cccc}
-0.5 & 0 & 1 &-2.5 \\ 
 0.5&  -0.5&  -1& 2\\ 
 -3&  -3&  -2&3 
\end{array} \right], ~y=\left[ \begin{array}{c}
1\\ 
-1\\ 
1
\end{array} \right],~b=\left[ \begin{array}{c}
-0.5\\ 
1\\ 
-1
\end{array} \right].$$}
\end{example}
It can be seen that $(0,0,2,1)^{T}$ and  $(0,1,-1/2,0)^{T}$  are  the sparsest solutions to the above convex system. Next, we show that the above two sparsest solutions satisfy some assumptions in Theorems \ref{infinite solution} and \ref{inf sol1}.

\noindent \textbf{(i)} $x=(0,0,2,1)^{T}$: We have $A_{S}=\left[ \begin{array}{cc}
-2& 5  \\ 
 4&  -9\\ 
 -2&  5 
\end{array} \right]$, $B_{I,S}=\left[ \begin{array}{cc}
1 & -2.5  \\ 
 -2&  3
\end{array} \right]$ and $B_{\bar{I},S}=\left[ \begin{array}{cc}
-1 & 2
\end{array} \right]$. We can see that $$\mathrm{Null}(A_{S})=\lbrace 0\rbrace, ~ \mathrm{Null}(B_{I,S})=\lbrace 0\rbrace, ~ \mathrm{Null}(B_{\bar{I},S})\neq \lbrace 0\rbrace,$$ and  $$(2,1)^{T}\in \lbrace d: B_{I,S}d< 0\rbrace\cap \mathrm{Null}(B_{\bar{I},S}), ~ (-2,-1)^{T}\in \lbrace d: B_{I,S}d> 0\rbrace\cap \mathrm{Null}(B_{\bar{I},S})$$ which satisfy   $(C4)$ and $(C3)$ in Theorem \ref{infinite solution}. The value of $\lambda$  in the proof of $(C4)$ or $(C3)$ can be determined, i.e., 
$$\lambda\in (0, 1/10\sqrt{3}] ~ \mathrm{for}~ (2, 1)^{T}, ~ \lambda\in [-1/10\sqrt{3}, 0) ~ \mathrm{for} ~ (-2, -1)^{T}.$$ Then   another sparsest solution can be formed as 
$$(0,0,2,1)^{T}+\lambda(0,0,2,1)^{T}, ~\lambda\in (0, 1/10\sqrt{3}],$$
and hence the  system $T$ in this example has infinitely many sparsest  solutions.

\noindent \textbf{(ii)} $x=(0,1,-1/2,0)^{T}$: We have $A_{S}=\left[ \begin{array}{cc}
0 & -2  \\ 
 1&  4\\ 
 0&  -2 
\end{array} \right]$, $B_{I,S}=(0,1)$ and $B_{\bar{I},S}=\left[ \begin{array}{cc}
-0.5 & -1  \\ 
 -3&  -2
\end{array} \right] $. It is easy to check $$\mathrm{Null}(A_{S})=\mathrm{Null}(B_{\bar{I},S})=\lbrace 0\rbrace~\mathrm{and}~\mathrm{Null}(B_{I,S})\neq \lbrace 0\rbrace$$ so that this example satisfies $\mathrm{Null}(M^{*})=\lbrace 0\rbrace$ and $\mathrm{Null}(B_{\bar{I},S})=\lbrace 0\rbrace$. We can find two vectors  which meet $(D5)$  and $(D4)$ in Theorem \ref{inf sol1}, i.e., $$(4,-1)^{T}\in \lbrace d: B_{I,S}d< 0\rbrace \cap \lbrace d: A_{S}d\neq 0\rbrace, ~(-4,1)^{T}\in \lbrace d: B_{I,S}d> 0\rbrace \cap \lbrace d: A_{S}d\neq 0\rbrace.$$ 
Then the value of  $\lambda$ in the proof of $(D5)$ or $(D4)$ can be determined. Analogously, for all $\lambda\in [\max(-1/10, -1/20\sqrt{3}), 0]$,  the vector $(0,1,-1/2,0)^{T}+\lambda(0,-4,1,0)^{T}$ is a  sparsest point in $T$. 
Note that $\mathrm{Null}(B_{I,S})\neq \lbrace 0\rbrace$, which also meets $(D3)$ in Theorem \ref{inf sol1}. We can find $(1,0)^{T}\in \mathrm{Null}(B_{I,S})$, and therefore $\lambda_{1}$ and $\lambda_{2}$ in the proof of $(D3)$ can be determined. Consequently,  for all $\lambda$ such that $\vert\lambda\vert\in [0, 1/10\sqrt{3}]$, the vector $(0,1,-1/2,0)^{T}+\lambda(0,1,0,0)^{T}$ is a sparsest point in $T$. 

\section{ Boundedness of the solution set of ($P_{0}$)}\label{the choice delta}

 In this section, some sufficient  conditions for   the boundedness of the solution set $\Lambda$ of ($P_{0}$) are also  identified.  We start to discuss the lower bound on the absolute value of nonzero components of vectors in $\Lambda$ given in \eqref{sparse solution set}.  We only consider the case that $\Lambda$ is bounded. 
\begin{lemma}\label{lemma lower bound 1}
Let  $k$ be the optimal value of \eqref{Pnew}. If  the  solution set $\Lambda$ is bounded,  then there exists a positive lower bound $\gamma^{*}$ for the  nonzero component $\vert x_{i} \vert$ of any vector $\vert x\vert, ~x\in \Lambda$, i.e.,
\begin{equation}\label{lower bound eq1}
\vert x_{i} \vert \geq \gamma^{*}, ~ i\in \mathrm{supp}(x).
\end{equation}
\end{lemma}
\begin{proof}
 We prove this result by considering only two situations: $\Lambda$ is finite or infinite.

$(i)$ Let the set $\Lambda$ be finite and bounded. Denote  the cardinality of $\Lambda$ as $L$ and  the sparsest solutions of \eqref{Pnew} as $\lbrace x^{p}\rbrace$, where $1\leqslant p\leqslant L$. Obviously, we can find the minimum value among the nonzero absolute entries of all vectors in $\Lambda$ and set such a minimal value as $\gamma^{*}$, which is expressed as 
$$\gamma^{*}=\min_{1\leq p \leq L} \min_{i\in \mathrm{supp}(x^{p})} \vert x_{i}^{p}\vert.$$
 This implies that  the absolute  values of the nonzero components of vectors in $\Lambda$ have a positive lower bound $\gamma^{*}$.

 $(ii)$ Let the set $\Lambda$ be infinite and bounded. In this case, $L$ is an infinite number. Since  $\Lambda$ is bounded,  there exists a positive number $U$ such that the absolute value of all entries of  vectors in $\Lambda$ is less or equal than $U$. We assume that \eqref{lower bound eq1} does not hold for $x\in \Lambda$. This means there  exists a sequence   $\lbrace x^{p}\rbrace\in \Lambda$, such that the minimum nonzero absolute entries of $x^{p}$ approach to $0$, i.e.,  
\begin{equation*}\label{eq1 lower bound}
\min_{ i\in \mathrm{supp}(x^{p})} \vert x_{i}^{p} \vert \rightarrow 0 ~~\mathrm{as}~~p\rightarrow \infty.
\end{equation*} 
 Since $\Lambda$ is bounded, this implies that   $$ \vert x^{p}_{i}\vert \leqslant U, ~ i\in \mathrm{supp}(x^{p}).$$  Following by Bolzano-Weierstrass Theorem,  the sequence $\lbrace x^{p}\rbrace$ has at least  one convergent subsequence, denoted still by $\lbrace x^{p}\rbrace$,   with a limit point $x^{*}\in T$ satisfying $\left\|x^{*}\right\|_{0}\leqslant k-1$.   This is a contradiction, and hence the lower bound is ensured when $\Lambda$ is infinite and bounded.  Combining $(i)$ and $(ii)$ obtains the desired result.
\end{proof}

 The above lemma  ensures the existence of a positive lower bound for the  absolute value of the nonzero components of the  vectors in $\Lambda$ when $\Lambda$ is bounded. 
In the following lemma, some sufficient conditions are developed to guarantee the boundedness of  $\Lambda$. 

\begin{lemma}\label{lower bound}
Let  $k$ be the optimal value of \eqref{Pnew}. The sparse solution set $\Lambda$ is bounded if one of the following  conditions holds:
  \begin{itemize}[leftmargin=0.2in]
  \item{$(E1)$}    For any $\Pi \subseteq \lbrace 1,\dots, n\rbrace$ and $\vert\Pi\vert= k $, we  have \begin{equation}\label{eq4 lower bound}
  \lbrace \eta: A_{\Pi}\eta=0 \rbrace \cap \lbrace \eta: B_{\Pi}\eta\leqslant 0 \rbrace=\lbrace 0\rbrace.
\end{equation}   
  \item{$(E2)$}   Any $k$ columns in $A$ are linearly independent.
  \item{$(E3)$}  $k< spark(A)$, where $spark(A)$ denote the minimum number of linearly dependent columns in $A$.
  \end{itemize}
 \end{lemma}
 
\begin{proof}
 First of all,  we suppose that the  set $\Lambda$ is unbounded. There exists a  sequence of the sparsest solutions of \eqref{Pnew}, denoted by  $\lbrace x^{p}\rbrace$, satisfying the following properties: $$\left\| x^{p}\right\|_{\infty}\rightarrow \infty~~\mathrm{as}~~p\rightarrow \infty$$ and there  is a fixed index set $S_{1}$ ($\vert S_{1}\vert \leq k$) such that   
\begin{equation*}\label{eq7 lower bound}
 \vert x_{i}^{p} \vert \rightarrow \infty~\mathrm{for}~\mathrm{all} ~i\in S_{1}, ~ \mathrm{as}~ p\rightarrow \infty 
\end{equation*} 
and the remaining components $x_{i}^{p},~ i\in S_{2}=\mathrm{supp}(x^{p})\setminus  S_{1}$ are bounded. 
Based on the fact that $x^{p}$ satisfies the  constraints in  \eqref{Pnew},  we have 
$$\norm{ A_{S_{2}}x^{p}_{S_{2}}+A_{S_{1}}x^{p}_{S_{1}}-y}_{2}\leqslant \epsilon, ~ B_{S_{2}}x^{p}_{S_{2}}+B_{S_{1}}x^{p}_{S_{1}}\leqslant b.$$
We divide  the above two inequalities by $\left\| x^{p}_{S_{1}}\right\|_{2}$ to obtain
$$\frac{\left\| A_{S_{2}}x^{p}_{S_{2}}+A_{S_{1}}x^{p}_{S_{1}}-y\right\|_{2}}{\left\| x^{p}_{S_{1}}\right\|_{2}}\leqslant \frac{\epsilon}{\left\| x^{p}_{S_{1}}\right\|_{2}}, ~ ~\frac{B_{S_{2}}x^{p}_{S_{2}}+B_{S_{1}}x^{p}_{S_{1}}}{\left\| x^{p}_{S_{1}}\right\|_{2}}\leqslant \frac{b}{\left\| x^{p}_{S_{1}}\right\|_{2}}.$$
Then we have
$$\norm{ A_{S_{2}}\frac{x^{p}_{S_{2}}}{\left\| x^{p}_{S_{1}}\right\|_{2}}+A_{S_{1}}\bar{\eta}-\frac{y}{\left\| x^{p}_{S_{1}}\right\|_{2}}}_{2}\leqslant \frac{\epsilon}{\left\| x^{p}_{S_{1}}\right\|_{2}}, ~~ B_{S_{2}}\frac{x^{p}_{S_{2}}}{\left\| x^{p}_{S_{1}}\right\|_{2}}+B_{S_{1}}\bar{\eta}\leqslant \frac{b}{\left\| x^{p}_{S_{1}}\right\|_{2}},$$ where $\bar{\eta}$ is a unit vector in $R^{\vert S_{1}\vert}$. Note that 
$$\lim_{p\rightarrow \infty} \frac{x^{p}_{S_{2}}}{\norm{x^{p}_{S_{1}}}_{2}}=0, ~ \lim_{p\rightarrow \infty}\frac{y}{\norm{x^{p}_{S_{1}}}_{2}}=0, ~ \lim_{p\rightarrow \infty}\frac{b}{\norm{x^{p}_{S_{1}}}_{2}}=0, ~ \lim_{p\rightarrow \infty}\frac{\epsilon}{\norm{x^{p}_{S_{1}}}_{2}}=0.$$ 
Thus there exists a unit vector 
$\bar{\eta}\in R^{\vert S_{1} \vert}$ satisfying 
\begin{equation*}
A_{S_{1}}\bar{\eta}=0, ~B_{S_{1}}\bar{\eta}\leqslant 0.
\end{equation*} 
This means 
$$\biggr\lbrace \eta: A_{S_{1}}\eta=0 \biggr\rbrace \cap \biggl\lbrace \eta: B_{S_{1}}\eta\leqslant 0\biggr\rbrace\neq \lbrace 0 \rbrace.$$
which contradicts to the assumption  \eqref{eq4 lower bound}.  Thus  under \eqref{eq4 lower bound},  $\Lambda$ is bounded.   It is clear that if any $k$ columns of $A$ are linearly independent or $k< spark(A)$, then the set $\lbrace \eta: A_{\Pi}\eta=0 \rbrace=\lbrace 0\rbrace$ and thus  \eqref{eq4 lower bound} holds. Hence the second and third conditions in Lemma \ref{lower bound} can also ensure $\Lambda$ to be bounded.
\end{proof}

\section{Conclusion}
In this paper,   some basic properties of the solutions of \eqref{Pnew} are developed such as the necessary conditions for a point being the sparsest point in the feasible set of \eqref{Pnew}.   
Some sufficient conditions   for  the nonuniqueness of the sparsest solutions of \eqref{Pnew} are also developed. 
We also discussed  the boundedness of the solution set of \eqref{Pnew} under  certain conditions. Based on this,  a  positive lower bound for the absolute nonzero  entries of the solutions to \eqref{Pnew} can be guaranteed  when the solution set of \eqref{Pnew} is bounded. 
These results can be applied to a class of $\ell_{0}$-problems such as the standard $\ell_{0}$-minimization problems (C1) and  (C2), and even some structured sparsity models.

\end{document}